\documentclass{amsart}
\usepackage{amssymb,latexsym}
\usepackage{amsfonts}
\usepackage{amsmath}
\usepackage[colorlinks,linkcolor=blue,anchorcolor=blue,citecolor=blue]{hyperref}

\newcommand\C{{\mathbb C}}

\newcommand\Q{{\mathbb Q}}

\newcommand\Z{{\mathbb Z}}

\newcommand\GL{{\mathrm {GL}}}
\newcommand\SL{{\mathrm {SL}}}
\newcommand\Gal{{\mathrm {Gal}}}

\newcommand\HH{{\mathcal H}}
\newcommand\OO{{\mathcal O}}
\newcommand\NN{{\mathcal{N}}}

\newcommand\ns{\mathrm{ns}}

\newcommand\ord{\mathrm{Ord}}
\newcommand\h{\mathrm{h}}

\newtheorem{theorem}{Theorem}[section]
\newtheorem{proposition}[theorem]{Proposition}
\newtheorem{lemma}[theorem]{Lemma}
\newtheorem{corollary}[theorem]{Corollary}

\newtheorem{remark}[theorem]{Remark}
\numberwithin{equation}{section}

\begin{document}
\title{Bounding the $j$-invariant of integral points on modular curves}


\author{Min Sha}
\address{Institut de Math\'ematiques de Bordeaux, Universit\'e Bordeaux 1, 33405 Talence Cedex, France}
\email{shamin2010@gmail.com}



\date{}

\thanks{
The author is supported by China Scholarship Council.}

\begin{abstract}
In this paper, we give some effective bounds for the $j$-invariant of integral points on arbitrary modular curves
 over arbitrary number fields assuming that the number of cusps is not less than 3.
\end{abstract}

\maketitle

\section{Introduction}
For a positive integer $N$, let $X(N)$ be the principal modular curve of level $N$. Let
$G$ be a subgroup of $\GL_{2}(\Z/N\Z)$ containing $-1$, and let $X_{G}$ be the
corresponding modular curve.
We denote by $\det G$ the image of $G$ under the determinant map $\det:\GL_{2}(\Z/N\Z)\to (\Z/N\Z)^{*}$.
This curve is defined over $\Q(\zeta_{N})^{\det G}$, where $\zeta_{N}=e^{2\pi i/N}$. So in particular it is defined over $\Q$
if $\det G=(\Z/N\Z)^{*}$.
We denote by~$j$ the standard $j$-invariant function on $X_{G}$.
We use the common notation $\nu_{\infty}(G)$ for the number of cusps of $X_{G}$.

Let $K_{0}$ be a number field containing $\Q(\zeta_{N})^{\det G}$. Then $X_{G}$ is defined over $K_{0}$.
Let $S_{0}$ be a finite set of absolute values of $K_{0}$,
containing all the Archimedean valuations and normalized with respect to $\Q$.
We call a $K_{0}$-rational point $P\in X_{G}(K_{0})$ an $S_{0}$-integral point if $j(P)\in \mathcal{O}_{S_{0}}$,
where $\mathcal{O}_{S_{0}}$ is the ring of $S_{0}$-integers in $K_{0}$.

By the classical Siegel's finiteness theorem~\cite{Si29}, $X_{G}$ has only finitely many $S_{0}$-integral points
when $X_{G}$ has positive genus or $\nu_{\infty}(G)\ge 3$. But the existing proofs of Siegel's theorem are not effective, that is they don't
 provide with any effective bounds for the $j$-invariant of $S_{0}$-integral points.

 Since 1995, Yuri Bilu and his collaborators have succeeded in getting
effective Siegel's theorem for various classes of modular curves. Bilu \cite[Proposition 5.1]{Bi95} showed that
the $j$-invariant of the $S_{0}$-integral points of $X_{G}$ can be effectively bounded provided that $\nu_{\infty}(G)\ge 3$,
but there was no quantitative version therein. Afterwards, Bilu \cite[Theorem 10]{Bilu02} proved that
the $j$-invariant of integral points of $X_{0}(N)$ could be effectively bounded if $N\not\in\{1,2,3,5,7,13\}$,
and Bilu and Illengo \cite{BI11} obtained similar results for ``almost every" modular curve. But they
still gave no quantitative results.

By using Runge's method, the first explicit bound for the $j$-invariant of the $S_{0}$-integral points of $X_{G}$
was given in \cite[Theorem 1.2]{BP10} when $X_{G}$ satisfies ``Runge condition" which roughly says that all the cusps are not conjugate.
When $G$ is the normalizer of a split Cartan subgroup of $\GL_{2}(\Z/p\Z)$, where $p$ is a prime number,
this bound can be sharply reduced, see \cite[Theorem 6.1]{BP10} and \cite[Theorem 1.1]{BP11}.
In particular, the authors in \cite{BP10,BP11,BPR11} showed various and interesting applications of
these bounds, such as rational points of modular curves \cite{BP10, BPR11}, Serre's uniformity problem
in Galois representation \cite{BP11}, and so on.

Most recently, without Runge condition and by using Baker's method, Bajolet and Sha \cite{BaSh1} gave an explicit bound
for the $j$-invariant of integral points on $X_{\ns}^{+}(p)$, which is
the modular curve of a prime level~$p$ corresponding to the normalizer of a non-split Cartan subgroup of $\GL_2(\Z/p\Z)$, $p\ge 7$.
Furthermore, a general method for computing integral points on $X_{\ns}^{+}(p)$ is developped in~\cite{BaBi11}.

In this paper, we apply Baker's method, based on Matveev \cite{Matveev} and Yu \cite{Yu07},
 to obtain some effective bounds for the $j$-invariant of the integral points on $X_{G}$
without assuming Runge condition but assuming that $\nu_{\infty}(G)\ge 3$.

We denote by $\h(\cdot)$ the usual absolute logarithmic height.
For $P\in X_{G}(\bar{\Q})$, we write $\h(P)=\h(j(P))$. Now we would like to state the main results.

\begin{theorem}\label{main1}
Assume that $K_{0}\subseteq\Q(\zeta_{N})$, $N$ is not a power of any prime, $\nu_{\infty}(G)\ge 3$,
and $S_{0}$ only consists of infinite places.
Then for any $S_{0}$-integral point $P$ on $X_{G}$, we have
$$
\h(P)\le C^{\varphi(N)}N^{\frac{3}{2}\varphi(N)+10}(\log N)^{\frac{5}{2}\varphi(N)-2},
$$
where $C$ is an absolute effective constant and $\varphi(N)$ is the Euler's totient function.
\end{theorem}

Actually, we obtain a more general Theorem \ref{main2} below, which applies to any number field and any ring of $S_{0}$-integers in it.

Put $d_{0}=[K_{0}:\Q]$ and $s_{0}=|S_{0}|$.
We define the following quantity
\begin{equation}\label{Delta}
\Delta=d_{0}^{-d_{0}}\sqrt{N^{d_{0}N}|D_{0}|^{\varphi(N)}}\left(\log(N^{d_{0}N}|D_{0}|^{\varphi(N)})\right)^{d_{0}\varphi(N)}
\left(\prod\limits_{\substack{v\in S_{0}\\v\nmid \infty}}\log\NN_{K_{0}/\Q}(v)\right)^{\varphi(N)},
\end{equation}
where $D_{0}$ is the absolute discriminant of $K_{0}$, and the norm of a finite place is, by definition, the
absolute norm of the corresponding prime ideal.
We denote by $p$ the maximal rational prime below $S_{0}$, with the convention $p=1$ if $S_{0}$ consists only
of the infinite places.

\begin{theorem}\label{main2}
Assume that $N$ is not a power of any prime and $\nu_{\infty}(G)\ge 3$.
Then for any $S_{0}$-integral point $P$ on $X_{G}$, we have
$$
\h(P)\le \left(Cd_{0}s_{0}N^{2}\right)^{2s_{0}N}(\log(d_{0}N))^{3s_{0}N}p^{d_{0}N}\Delta,
$$
where $C$ is an absolute effective constant.
\end{theorem}

The situation is different when $N$ is a prime power, see Section \ref{prime power}.  In this case we define
\begin{equation}
M=\left\{ \begin{array}{ll}
                 2N & \textrm{if $N$ is not a power of 2},\\
                  \\
                 3N & \textrm{if $N$ is a power of 2}.
                 \end{array} \right.
\notag
\end{equation}
Notice that $X_{G}$ is also a modular curve of level $M$.
\begin{theorem}\label{main3}
Assume that $N$ is a power of some prime and $\nu_{\infty}(G)\ge 3$.
Then for any $S_{0}$-integral point $P$ on $X_{G}$, we can get two upper bounds for
$\h(P)$ by replacing $N$ by $M$ in Theorem \ref{main1} and \ref{main2}.
\end{theorem}

\section{Notations and conventions}

Through out this paper, $\log$ stands for two different objects without confusion according to the context. One is the principal branch of the complex logarithm, in this case will use the following estimate without special reference
$$
|\log(1+z)|\le\frac{|\log(1-r)|}{r}|z|,
$$
for $|z|\le r<1$, see \cite[Formula (4)]{BP10}. The other one is the $p$-adic logarithm function, for example see \cite[Chapter IV Section 2]{Koblitz}.

Let $\mathcal{H}$ denote the Poincar$\acute{\rm e}$ upper half-plane: $\mathcal{H}=\{\tau\in\C: {\rm Im} \tau>0\}$.
For $\tau\in\mathcal{H}$, put $q_{\tau}=e^{2\pi i\tau}$.
We also put $\bar{\HH}=\HH\cup\Q\cup\{i\infty\}$. If $\Gamma$ is the pullback of $G\cap\SL_{2}(\Z/N\Z)$ to $\SL_{2}(\Z)$,
then the set $X_{G}(\C)$ of complex points is analytically isomorphic to the quotient $\bar{\HH}/\Gamma$,
supplied with the properly defined topology and analytic structure.
Moreover, the modular invariant $j$ defines a non-constant rational function on $X_{G}$, whose poles are exactly the cusps.
See any standard reference like \cite{La76, Shimura} for all the missing details.

For  $\textbf{a}=(a_{1},a_{2})\in \Q^{2}$,  we put $\ell_{\textbf{a}}
=B_{2}(a_{1}-\lfloor a_{1} \rfloor)/2$, where  $B_{2}(T)=T^{2}-T+\frac{1}{6}$ is the second Bernoulli polynomial and
$\lfloor a_{1} \rfloor$ is the largest integer not greater than $a_1$.
Obviously $|\ell_{\textbf{a}}|\le 1/12$, this will be used without special reference.

Let $\mathcal{A}_{N}$ be the subset of abelian group $(N^{-1}\Z/\Z)^{2}$ consisting of the elements with exact order $N$.
Obviously,
 $$
 |\mathcal{A}_{N}|=N^{2}\prod\limits_{p|N}(1-p^{-2})<N^{2},
 $$
 the product runs through all primes dividing $N$.
 Moreover we always choose a representative element of
 ${\textbf{a}=\left(a_{1},a_{2}\right)}\in (N^{-1}\Z/\Z)^{2}$ satisfying ${0 \le a_{1},a_{2} <1}$.
 So in the sequel for every $\textbf{a}\in (N^{-1}\Z/\Z)^{2}$, we have $\ell_{\textbf{a}}=B_{2}(a_{1})/2$.

Throughout this paper, we fix an algebraic closure $\bar{\Q}$ of $\Q$, which is assumed to be a subfield of $\C$.
Every number field used in this paper is presumed to be a subfield of $\bar{\Q}$.

For a number field $K$, we denote by $M_{K}$ the set of all valuations (or places) of $K$ extending the standard infinite and
$p$-adic valuations of $\Q$: $|2|_{v}=2$ if $v\in M_{K}$ is infinite, and $|p|_{v}=p^{-1}$ if $v$ extends the $p$-adic
valuation of $\Q$. We denote by $M_{K}^{\infty}$ and $M_{K}^{0}$ the subsets of $M_{K}$ consisting of the infinite
(Archimedean) and the finite (non-Archimedean) valuations, respectively.

Given a number field $K$ of degree $d$, for any $v\in M_{K}$, $K_{v}$ is the completion of $K$ with respect to the valuation $v$ and $\bar{K}_{v}$
its algebraic closure. We still denote by $v$ the unique extension of $v$ in $\bar{K}_{v}$.
Let $d_{v}=[K_{v}:\Q_{v}]$ be the local degree of $v$.

For a number field $K$ of degree $d$, the absolute logarithmic height of an algebraic number $\alpha\in K$ is
defined by $\h(\alpha)=d^{-1}\sum_{v\in M_{K}}d_{v}\log^{+}|\alpha|_{v}$, where $\log^{+}|\alpha|_{v}=\log\max\{|\alpha|_{v},1\}$.

Throughout the paper the symbol $\ll$ implies an absolute effective constant.
We also use the notation $O_{v}(\cdot)$. Precisely, $A=O_{v}(B)$ means that $|A|_{v}\le B$.

\section{Preparations}

In this section, we assume that $N\ge 2$.

\subsection{Siegel functions}
Let $\textbf{a}=(a_{1},a_{2})\in \Q^{2}$ be such that $\textbf{a}\not\in\Z^{2}$, and let $g_{\textbf{a}}: \HH\to\C$
 be the corresponding
\textit{Siegel function}, see \cite[Section 2.1]{KL81}. We have the following infinite product presentation for $g_{\textbf{a}}$,
see \cite[Formula (7)]{BP10},
$$
g_{\textbf{a}}(q_{\tau})=-q_{\tau}^{B_{2}(a_{1})/2}e^{\pi ia_{2}(a_{1}-1)}\prod\limits_{n=0}^{\infty}(1-q_{\tau}^{n+a_{1}}e^{2\pi ia_{2}})(1-q_{\tau}^{n+1-a_{1}}e^{-2\pi ia_{2}}).
$$

For the elementary properties of $g_{\textbf{a}}$, please see \cite[Pages 27-31]{KL81}.
Especially, the order of vanishing of $g_{\textbf{a}}$ at $i\infty$ (i.e., the only rational number $\ell$
such that the limit $\lim_{\tau\to i\infty}q_{\tau}^{-\ell}g_{\textbf{a}}$ exists and is non-zero) is equal to $\ell_{\textbf{a}}$.

For a number field $K$ and $v\in M_{K}$, we define $g_{\textbf{a}}(q)$ as the above, where $q\in \bar{K}_{v}$ satisfies $|q|_{v}<1$.
Notice that here we should fix $q^{1/(12N^{2})}\in \bar{K}_{v}$, then everything is well defined.

Given two positive integers $k$ and $\ell$, we denote by $P_{k}$ the set of partitions of $k$ into positive summands, and let
 $p_{\ell}(k)$ be the number of partitions of $k$ into exactly $\ell$ positive summands.
By \cite[Theorem 14.5]{Apostol}, we easily get
$$
|P_{k}|<e^{k/2}, \quad k\ge 64.
$$
Then according to the table of partitions or computer calculations, we can obtain
$$
|P_{k}|<e^{k/2}, \quad k\ge 1.
$$
 \begin{proposition}\label{ga1}
Let ${\rm\bf a}\in \mathcal{A}_{N}$. If $q\in \bar{K}_{v}$ satisfies $|q|_{v}<1$, then we have
\begin{equation}
-q^{-\ell_{\rm\bf a}}\gamma_{{\rm\bf a}}^{-1}g_{{\rm\bf a}}(q)=1+\sum\limits_{k=1}^{\infty}\phi_{{\rm\bf a}}(k)q^{k/N},
\notag
\end{equation}
where
\begin{equation}
\gamma_{{\rm\bf a}}=\left\{ \begin{array}{ll}
                 e^{\pi ia_{2}(a_{1}-1)} & \textrm{if $a_{1}\ne 0$},\\
                 e^{-\pi ia_{2}}(1-e^{2\pi ia_{2}}) & \textrm{if $a_{1}=0$};
                 \end{array} \right.
\notag
\end{equation}
and
$$
\phi_{{\rm\bf a}}(k)=\sum\limits_{\ell\in S^{1}_{{\rm\bf a}k}}m_{\ell}(-e^{2\pi ia_{2}})^{\ell}+\sum\limits_{\ell\in S^{2}_{{\rm\bf a}k}}m_{\ell}^{\prime}(-e^{-2\pi ia_{2}})^{\ell}
+\sum\limits_{\ell\in S^{3}_{{\rm\bf a}k}}\sum\limits_{(\ell_{1},\ell_{2})\in T^{\ell}_{{\rm\bf a}k}}m_{\ell_{1}\ell_{2}}(-e^{2\pi ia_{2}})^{\ell_{1}}(-e^{-2\pi ia_{2}})^{\ell_{2}},
$$
where $S^{1}_{{\rm\bf a}k}$, $S^{2}_{{\rm\bf a}k}$ and $S^{3}_{{\rm\bf a}k}$ are three subsets of $\{1,2,\cdots,\lfloor k/N\rfloor+1\}$,
$T^{\ell}_{{\rm\bf a}k}$ is a subset of $\{(\ell_{1},\ell_{2}):1\le \ell_{1},\ell_{2}\le\lfloor k/N\rfloor+1,\ell_{1}+\ell_{2}=\ell\}$,
and $m_{\ell},m_{\ell}^{\prime}, m_{\ell_{1}\ell_{2}}$ are some positive integers.
In particular, we have
$$
|\phi_{{\rm\bf a}}(k)|_{v}\le e^{k}.
$$
\end{proposition}
\begin{proof}
In this proof, we fix an integer $k\ge 1$.

Suppose that $a_{1}=k_{1}/N$ with $0\le k_{1}\le N-1 $.
Let $S_{1}=\{nN+k_{1}:0\le n\le \lfloor k/N\rfloor, 0<nN+k_{1}\le k\}$ and
$S_{2}=\{nN+N-k_{1}:0\le n\le \lfloor k/N\rfloor, nN+N-k_{1}\le k\}$.
It is easy to see that if $k_{1}=0$ or $N/2$, then $S_{1}=S_{2}$; otherwise $S_{1}\cap S_{2}= \emptyset$.

Notice that the coefficient $\phi_{{\rm\bf a}}(k)$ of $q^{k/N}$ equals to the coefficient of $q^{k}$
in the expansion of the following finite product,
\begin{equation}\label{expansion}
\prod\limits_{n\in S_{1}}(1-q^{n}e^{2\pi ia_{2}})\prod\limits_{n\in S_{2}}(1-q^{n}e^{-2\pi ia_{2}}).
\end{equation}
If $S_{1}$ and $S_{2}$ are both empty, then the coefficient $\phi_{{\rm\bf a}}(k)=0$.

We say $\ell\in S^{1}_{{\rm\bf a}k}$ if and only if there exist $\ell$ positive integers in
$S_{1}$ such that the sum of them equals to $k$,
and let $m_{\ell}$ count the number of different ways.
Similarly for the definitions of $S^{2}_{{\rm\bf a}k}$ and $m_{\ell}^{\prime}$.

We say $\ell\in S^{3}_{{\rm\bf a}k}$ if and only if there exist $\ell_{1}$ positive integers in $S_{1}$
and $\ell_{2}$ positive integers in
$S_{2}$ such that
the sum of them equals to $k$, then
$(\ell_{1},\ell_{2})\in T^{\ell}_{{\rm\bf a}k}$ and let $m_{\ell_{1}\ell_{2}}$ count the number of different ways.

Then the desired expression of $\phi_{{\rm\bf a}}(k)$ follows easily from the definitions.

For each element $x\in P_{k}$, let $m_{x}$ be the number of the times of
$x$ appearing in the expansion of (\ref{expansion}). Then we obtain
\begin{align*}
|\phi_{{\rm\bf a}}(k)|_{v}&\le\sum\limits_{\ell\in S^{1}_{{\rm\bf a}k}}m_{\ell}+\sum\limits_{\ell\in S^{2}_{{\rm\bf a}k}}m_{\ell}^{\prime}+
\sum\limits_{\ell\in S^{3}_{{\rm\bf a}k}}\sum\limits_{(\ell_{1},\ell_{2})\in T^{\ell}_{{\rm\bf a}k}}m_{\ell_{1}\ell_{2}}\\
&= \sum\limits_{x\in P_{k}}m_{x}.
\end{align*}

If $k_{1}\ne 0$ and $N/2$, then $S_{1}\cap S_{2}=\emptyset$. So for each $x\in P_{k}$, we have $m_{x}=0$ or 1.
Hence,
$$
\sum\limits_{x\in P_{k}}m_{x} \le |P_{k}|<e^{k/2}.
$$

If $k_{1}=0$ or $N/2$, then $S_{1}=S_{2}$. Suppose that $\lfloor k/N\rfloor\ge 3$.
Given $x\in P_{k}$ with $\ell$ entries, if $\ell\le \lfloor k/N\rfloor$, then we have $m_{x}\le 2^{\ell}$; otherwise we have $m_{x}=0$.
Hence,
$$
\sum\limits_{x\in P_{k}}m_{x} \le
\sum\limits_{\ell\le \lfloor k/N\rfloor}2^{\ell}p_{\ell}(k)\le 2^{\lfloor k/N\rfloor}|P_{k}|<e^{k}.
$$
If $\lfloor k/N\rfloor\le 2$, one can verify the inequality by explicit computations.
\end{proof}

\subsection{Modular units on $X(N)$}

Recall that by a modular unit on a modular curve we mean that a rational function having poles and zeros only at the cusps.

For $\textbf{a}\in (N^{-1}\Z/\Z)^{2}$, we denote $g_{\textbf{a}}^{12N}$ by $u_{\textbf{a}}$, which
is a modular unit on  $X(N)$.
Moreover, we have $u_{\textbf{a}}=u_{\textbf{a}^{\prime}}$ when $\textbf{a}\equiv \textbf{a}^{\prime}$ mod $\Z^{2}$.
Hence, $u_{\textbf{a}}$ is well-defined when ${\textbf{a}}$ is a non-zero element of the abelian group $(N^{-1}\Z/\Z)^{2}$.
Moreover, $u_{\textbf{a}}$ is integral over $\Z[j]$.
For more details, see \cite[Section 4.2]{BP10}.

Furthermore, the Galois action on the set $\{u_{\textbf{a}}\}$ is compatible with the right linear action of
$\GL_{2}(\Z/N\Z)$ on it. That is, for any $\sigma\in\Gal(\Q(X(N))/\Q(j))=\GL_{2}(\Z/N\Z)/\pm 1$ and any $\textbf{a}\in(N^{-1}\Z/\Z)^{2}$, we have
$$
u_{\rm\bf a}^{\sigma}=u_{\rm\bf a\sigma}.
$$

Here we borrow a result and its proof from \cite{BaBi11} for subsequent applications and the conveniences of readers.
\begin{proposition}[\cite{BaBi11}]\label{u1}
We have
\begin{equation}
\prod\limits_{\textbf{a}\in\mathcal{A}_{N}}u_{\textbf{a}}=\pm\Phi_{N}(1)^{12N}
=\left\{ \begin{array}{ll}

                 \pm \ell^{12N} & \textrm{ if $N$ is a power of a prime $\ell$},\\

                 \pm 1 & \textrm{ if $N$ has at least two distinct prime factors},
                 \end{array} \right.
\notag
\end{equation}
where $\Phi_{N}$ is the $N$-th cyclotomic polynomial.
\end{proposition}
\begin{proof}
We denote by $u$ the left-hand side of the equality.
Since the set $\mathcal{A}_{N}$ is stable with respect to $\GL_{2}(\Z/N\Z)$, $u$
is stable with respect to the Galois action over the field $\Q(X(1))=\Q(j)$. So $u\in \Q(j)$.
Moreover, since $u$ is integral over $\Z[j]$, $u\in \Z[j]$. Notice that
$X(1)$ has only one cusp and
$u$ has no zeros and poles outside the cusps,
so we must have $u$ is a constant and $u\in\Z$.

Furthermore, we have
\begin{align*}
u&=\prod\limits_{(a_{1},a_{2})\in\mathcal{A}_{N}}q^{6NB_{2}(a_{1})}e^{12N\pi ia_{2}(a_{1}-1)}\prod\limits_{n=0}^{\infty}(1-q^{n+a_{1}}e^{2\pi ia_{2}})^{12N}
(1-q^{n+1-a_{1}}e^{-2\pi ia_{2}})^{12N}\\
&\stackrel{q=0}{\mathrel{=\!=\!=}}\pm \prod\limits_{\substack{(a_{1},a_{2})\in\mathcal{A}_{N}\\a_{1}=0}}(1-e^{2\pi ia_{2}})^{12N}\\
&=\pm\prod\limits_{\substack{1\le k<N\\ \gcd(k,N)=1}}(1-e^{2k\pi i/N})^{12N}\\
&=\pm\Phi_{N}(1)^{12N}.
\end{align*}
\end{proof}

\subsection{$X_{G}$ and $X_{G_{1}}$} \label{X_{G_{1}}}

Let $G_{1}=G\cap\SL_{2}(\Z/N\Z)$ and $X_{G_{1}}$ be the modular curve corresponding to $G_{1}$.
In this subsection, we assume that $X_{G_{1}}$ is defined over a number field $K$. Then $X_{G}$ is also defined over $K$.
Since $X_{G}$ and $X_{G_{1}}$ have the same geometrically integral model,
every $K$-rational point of $X_{G}$ is also a $K$-rational point of $X_{G_{1}}$.

For each cusp $c$ of $X_{G_{1}}$, let $t_{c}$ be its local parameter constructed in \cite[Section 3]{BP10}.
Put $q_{c}=t_{c}^{e_{c}}$, where $e_{c}$ is the ramification index of the natural covering $X_{G_{1}}\to X(1)$ at $c$. Notice that $e_{c}|N$.
Furthermore, for any $v\in M_{K}$, let $\Omega_{c,v}$ be the set constructed in \cite[Section 3]{BP10}
on which $t_{c}$ and $q_{c}$ are defined and analytic. Here we quote \cite[Proposition 3.1]{BP10} as following.
\begin{proposition}[\cite{BP10}]\label{nearby}
Put
\begin{equation}
X_{G_{1}}(\bar{K}_{v})^{+}
=\left\{ \begin{array}{ll}

                 \{P\in X_{G_{1}}(\bar{K}_{v}): |j(P)|_{v}>3500\} & \textrm{ if $v \in M_{K}^{\infty}$},\\

                 \{P\in X_{G_{1}}(\bar{K}_{v}): |j(P)|_{v}>1\} & \textrm{ if $v \in M_{K}^{0}$}.
                 \end{array} \right.
\notag
\end{equation}
Then
$$
X_{G_{1}}(\bar{K}_{v})^{+}\subseteq \bigcup\limits_{c}\Omega_{c,v}
$$
with equality for the non-Archimedean $v$, where the union runs through all the cusps of $X_{G_{1}}$.
Moreover, for $P\in \Omega_{c,v}$ we have
\begin{equation}\label{nearby2}
\frac{1}{2}|j(P)|_{v}\le |q_{c}(P)^{-1}|_{v}\le \frac{3}{2}|j(P)|_{v}
\end{equation}
if $v$ is Archimedean, and $|j(P)|_{v}=|q_{c}(P)^{-1}|_{v}$ if $v$ is non-Archimedean.
\end{proposition}

We will use the above proposition several times without special reference. Moreover, this proposition implies that
for every $P\in X_{G_{1}}(\bar{K}_{v})^{+}$ there exists a cusp $c$ such that $P\in\Omega_{c,v}$. We call $c$ a $v$-nearby cusp of $P$.

We get directly the following corollary from Proposition \ref{ga1}.
\begin{corollary}\label{ga2}
Let $c$ be a cusp of $X_{G_{1}}$, $v\in M_{K}$ and $P\in \Omega_{c,v}$. Assume that $|q_{c}(P)|_{v}\le 10^{-N}$.
For ${\rm\bf a}\in \mathcal{A}_{N}$, we have
\begin{equation}
-q_{c}^{-\ell_{\rm\bf a}}\gamma_{{\rm\bf a}}^{-1}g_{{\rm\bf a}}(q_{c}(P))=1+O_{v}(4|q_{c}(P)|_{v}^{1/N}).
\notag
\end{equation}
\end{corollary}

The following proposition follows directly from \cite[Propositions 2.3 and 2.5]{BP10}.
\begin{proposition}\label{log1}
Let $c$ be a cusp of $X_{G_{1}}$, $v\in M_{K}$ and $P\in \Omega_{c,v}$. For every ${\rm\bf a}\in \mathcal{A}_{N}$,  we have
\begin{equation}
\left|\log|g_{{\rm\bf a}}(q_{c}(P))|_{v}-\ell_{\rm\bf a}\log|q_{c}(P)|_{v}\right|
\left\{ \begin{array}{ll}

                 \le \log N & \textrm{ if $v\in M_{K}^{\infty}$},\\

                 =0 & \textrm{ if $|N|_{v}=1$},\\

                 \le \frac{\log \ell}{\ell-1} & \textrm{ if $v|\ell|N$},
                 \end{array} \right.
\notag
\end{equation}
where $\ell$ is some prime factor of $N$.
\end{proposition}

\subsection{Modular units on $X_{G_1}$}

We apply the notations in the above subsection.

We denote by $\mathcal{M}_{N}$ the set of elements of exact order $N$ in $(\Z/N\Z)^{2}$.
Let us consider the natural right group action of $G_{1}$ on $\mathcal{M}_{N}$. 
Following the proof of \cite[Lemma 2.3]{BI11}, we see that the number of the orbits of $\mathcal{M}_{N}/G_{1}$ is
equal to $\nu_{\infty}(G)$.

Obviously, 
 when we consider the natural right group action $\mathcal{A}_{N}/G_{1}$,
 there are also $\nu_{\infty}(G)$ orbits of this group action. So
 $$
 \nu_{\infty}(G)\le|\mathcal{A}_{N}|<N^{2}.
 $$

Let $T$ be any subset of $\mathcal{A}_{N}$, we define
$$
u_{T}=\prod\limits_{\textbf{a}\in T}u_{\textbf{a}}.
$$
Let $\OO$ be an orbit of the right group action $\mathcal{A}_{N}/G_{1}$, we have
\begin{equation}\label{uo}
u_{\OO}=\prod\limits_{\textbf{a}\in\OO}u_{\textbf{a}}.
\end{equation}
By \cite[Proposition 4.2 (ii)]{BP10}, $u_{\OO}$ is a rational function on the modular curve $X_{G_{1}}$.
In fact, $u_{\OO}$ is a modular unit on $X_{G_{1}}$.

For any cusp $c$, we denote by $\ord_{c}(u_{\OO})$ the vanishing order of $u_{\OO}$ at $c$.
For $v\in M_{K}$, define
\begin{equation}
\rho_{v}=
\left\{ \begin{array}{ll}

                  12N^{3}\log N & \textrm{ if $v \in M_{K}^{\infty}$},\\

                 0 & \textrm{ if $v \in M_{K}^{0}$ and $|N|_{v}=1$},\\

                  \frac{12N^{3}\log \ell}{\ell-1} & \textrm{ if $v \in M_{K}^{0}$ and $v|\ell|N$},
                 \end{array} \right.
\notag
\end{equation}
where $\ell$ is some prime factor of $N$.

Then $u_{\OO}$ has the following properties.
\begin{proposition} \label{munit}
{\rm (i)} Put $\lambda=(1-\zeta_{N})^{12N^{3}}$. Then the functions $u_{\OO}$ and $\lambda u_{\OO}^{-1}$ are integral
over $\Z[j]$.

{\rm (ii)} For the cusp $c_{\infty}$ at infinity, we have
$$
\ord_{c_\infty}(u_{\OO})=12Ne_{c_\infty}\sum\limits_{{\rm\bf a}\in \OO}\ell_{{\rm\bf a}}.
$$
For any cusp $c$, we have $|\ord_{c}(u_{\OO})|<N^{4}$.

{\rm (iii)} Let $c$ be a cusp of $X_{G_{1}}$, $v\in M_{K}$ and $P\in \Omega_{c,v}$. Assume that $|q_{c}(P)|_{v}\le 10^{-N}$.
Then we have
\begin{equation}
q_{c}(P)^{-\ord_{c}(u_{\OO})/e_{c}}\gamma_{\OO,c}^{-1}u_{\OO}(P)=1+O_{v}(4^{12N^{3}}|q_{c}(P)|_{v}^{1/N}),
\notag
\end{equation}
where $\gamma_{\OO,c}\in\Q(\zeta_{N})$ and $\h(\gamma_{\OO,c})\le 12N^{3}\log 2$.

{\rm (iv)} Let $c$ be a cusp of $X_{G_{1}}$ and $v\in M_{K}$. For $P\in \Omega_{c,v}$, we have
$$
\left|\log|u_{\OO}(P)|_{v}-\frac{\ord_{c}(u_{\OO})}{e_{c}}\log|q_{c}(P)|_{v}\right|\le\rho_{v}.
$$

{\rm (v)} For $v\in M_{K}^{\infty}$ and $P\in X_{G_{1}}(K_{v})$, we have
$$
\left|\log|u_{\OO}(P)|_{v}\right|
\le N^{3}\log(|j(P)|_{v}+2400)+\rho_{v}.
$$

{\rm (vi)} The group generated by the principal divisor $(u_{\OO})$, where $\OO$ runs over the orbits of $\mathcal{A}_{N}/G_{1}$,
is of rank $\nu_{\infty}(G)-1$.
\end{proposition}
\begin{proof}
(i) See \cite[Proposition 4.2 (i)]{BP10}.

(ii) Similar to the proof of \cite[Proposition 4.2 (iii)]{BP10}. The $q$-order of vanishing of $u_{\OO}$ at $i\infty$ is
$12N\sum\limits_{{\rm\bf a}\in \OO}\ell_{{\rm\bf a}}$. Then
$$
\ord_{c_\infty}(u_{\OO})=12Ne_{c_\infty}\sum\limits_{{\rm\bf a}\in \OO}\ell_{{\rm\bf a}}.
$$
Since $|\ell_{{\rm\bf a}}|\le 1/12$, we have $|\ord_{c_\infty}(u_{\OO})|\le Ne_{c_\infty}|\OO|<N^{4}$.
The case of arbitrary $c$ reduces to the case $c=c_\infty$ by replacing $\OO$ by $\OO\sigma$ where
$\sigma\in\GL_{2}(\Z/N\Z)$ is such that $\sigma(c)=c_\infty$.

(iii) Similar to the proof of \cite[Proposition 4.4]{BP10} by using Corollary \ref{ga2} except for
the height of $\gamma_{\OO}$. In fact, if $c=c_\infty$, we have
$\gamma_{\OO,c}=\prod\limits_{\rm\bf a\in\OO}\gamma_{\rm\bf a}^{12N}$. Then
$\h(\gamma_{\OO,c})\le 12N\sum\limits_{\rm\bf a\in\OO}\h(\gamma_{\rm\bf a})\le 12N|\OO|\log 2<12N^{3}\log 2$.
The general case reduces to the case $c=c_\infty$ by applying a suitable Galois automorphism.

(iv) and (v) They follow from \cite[Proposition 4.4]{BP10}.

(vi) By Proposition \ref{u1}, the rank of the free abelian group $(u_{\OO})$ is at most $\nu_{\infty}(G)-1$. Then
Manin-Drinfeld theorem, as stated in \cite{KL81}, tells us that this rank is maximal possible.
\end{proof}

\section{Siegel's theory of convenient units}
We recall here Siegel's construction \cite{Si69} of convenient units in a number field $K$ of degree $d$, in the form adapted
to the needs of the present paper.
The results of this section are well-known, but not always in the set-up we wish them to have.

Let $S$ be a finite set of absolute values of $K$,
containing all the Archimedean valuations and normalized with respect to $\Q$.
Fix a valuation $v_{0}\in S$, we put
$$
S^{\prime}=S\setminus \{v_{0}\},\quad s=|S|\ge 2, \quad r=s-1,\quad d^{\prime}=\max\{d,3\},\quad \zeta=1201\left(\frac{\log d^{\prime}}{\log\log d^{\prime}}\right)^{3}.
$$
Let $\xi_{1},\cdots,\xi_{r}$ be a fundamental system of $S$-units. The $S$-regulator $R(S)$ is the absolute value of
the determinant of the $r\times r$ matrix
\begin{equation}\label{regulator}
(d_{v}\log|\xi_{k}|_{v})_{\substack{v\in S^{\prime}\\1\le k\le r}}
\end{equation}
(we fix some ordering for the set $S^{\prime}$), where $d_{v}=[K_{v}:\Q_{v}]$ is the local degree of $v$.
It is well-defined and is equal to the usual regulator $R_{K}$ when $S$ is the set of infinite places.

\begin{proposition}\label{S-unit}
There exists a fundamental system of $S$-units $\eta_{1},\cdots,\eta_{r}$ satisfying
\begin{align*}
&\h(\eta_{1})\cdots\h(\eta_{r})\le d^{-r}r^{2r}R(S),\\
&(\zeta d)^{-1}\le \h(\eta_{k})\le d^{-1}r^{2r}\zeta^{r-1}R(S)\quad (k=1,\cdots,r).
\end{align*}
Furthermore, the entries of the inverse matrix of (\ref{regulator}) are bounded in absolute value by $r^{2r}\zeta$.
\end{proposition}
\begin{proof}
See \cite[Lemma 1]{Bugeaud1}. Notice that the left-hand inequality in the second inequality is a \
well-known result of Dobrowolski \cite{Dobrowolski1}.
\end{proof}

\begin{corollary}\label{eta}
For the unit $\eta=\eta_{1}^{b_{1}}\cdots\eta_{r}^{b_{r}}$, where $\eta_{1},\cdots,\eta_{r}$ are from Proposition \ref{S-unit}
and $b_{1},\cdots, b_{r}\in\Z$, put $B^{*}=\max\{|b_{1}|,\cdots,|b_{r}|\}$, then we have
\begin{align*}
&\h(\eta)\le d^{-1}r^{2r+1}\zeta^{r-1}B^{*}R(S),\\
&B^{*}\le 2dr^{2r}\zeta \h(\eta).
\end{align*}
\end{corollary}
\begin{proof}
The first inequality follows from Proposition \ref{S-unit} and standard height estimates.

Write
$$
d_{v}\log |\eta|_{v}=\sum\limits_{k=1}^{r}d_{v}b_{k}\log |\eta_{k}|_{v}, \quad v\in S^{\prime}.
$$
Resolving this in terms of $b_{1},\cdots,b_{r}$ and using the final statement of Proposition \ref{S-unit},
we obtain
$$
B^{*}\le r^{2r}\zeta\sum\limits_{v\in S^{\prime}}d_{v}|\log |\eta|_{v}|\le r^{2r}\zeta\sum\limits_{v\in S}d_{v}|\log |\eta|_{v}|.
$$
Since $\eta$ is an $S$-unit,
$$
\sum\limits_{v\in S}d_{v}|\log |\eta|_{v}|=d(\h(\eta)+\h(\eta^{-1}))=2d\h(\eta).
$$
Then the corollary is proved.
\end{proof}

Finally, we quote two estimates of the $S$-regulator in terms of the usual regulator $R_{K}$,
the class number $h_{K}$, the degree $d$ and the discriminant $D$ of the field $K$.

\begin{proposition}\label{S-regulator}
We have
\begin{gather}
0.1\le R(S) \leq h_{K}R_{K} \prod\limits_{\substack{v\in S\\v\nmid\infty}} \log \NN(v),\notag\\
R(S)\ll
d^{-d} \sqrt {|D|} (\log |D|)^{d-1} \prod\limits_{\substack{v\in S\\v\nmid\infty}} \log \NN(v).
\notag
\end{gather}
\end{proposition}

For the first inequality see  \cite[Lemma~3]{Bugeaud1}; one may remark that the lower bound ${R(S)\ge 0.1}$ follows from Friedman's famous lower bound~\cite{Fr89} for the usual regulator ${R_{K}\ge 0.2}$.
The second one follows from  Siegel's estimate \cite[Satz~1]{Si69}
$$
h_{K}R_{K} \ll d^{-d} \sqrt {|D|} (\log |D|)^{d-1};
$$
in fact there is an explicit bound for $h_{K}R_{K}$ therein.

\section{Baker's inequality}
In this section we state Baker's inequality, which is the main technical tool of the proof. It is actually an adaptation of a result in \cite{ABBN}.
For the convenience of readers, we also quote its proof with slight change.

For a number field $K$ and $v\in M_{K}$, we denote by $p_{v}$ the underlying prime of~$v$ when~$v$ is non-archimedean. Next, we let
\begin{itemize}


\item
$\theta_0, \theta_1, \cdots, \theta_r$ be non-zero algebraic numbers, belonging to $K$;

\item
$\Theta_0, \Theta_1, \cdots, \Theta_r$ be real numbers satisfying
\begin{align*}
&\Theta_k\ge \max\left\{d\h(\theta_k), 1\right\} \quad (k=0,1, \cdots ,r);
\end{align*}

\item
$b_1, \ldots, b_r$ be rational integers, ${\Lambda =\theta_0\theta_1^{b_1}\cdots\theta_r^{b_r}}$,
$B^{*}=\{|b_{1}|,|b_{2}|,\cdots,|b_{r}|\}$.

\end{itemize}

\begin{theorem}[\cite{ABBN}]
\label{tbaker}
There exists an absolute constant~$C$ that can be determined explicitly such that the following holds.
Assume that ${\Lambda\ne 1}$. Then for any real number $B$ satisfying $B\ge B^{*}$ and $B\ge \max\{3, \Theta_1,\cdots,\Theta_r\}$, we have
$$
\left|\Lambda-1\right|_{v}\ge e^{-\Upsilon  \Theta_0\Theta_1\cdots \Theta_r\log B},
$$
where
$$
\Upsilon =
\begin{cases}
C^r d^2\log(2d), & v\mid \infty,\\
(Cd)^{2r+6}p_{v}^d, & v|p_{v}<\infty.
\end{cases}
$$
\end{theorem}

\begin{proof}
The Archimedean case is due to Matveev,  see Corollary~2.3 from~\cite{Matveev}. We use this result with ${n=r+1}$, with
${1, b_1, \ldots, b_r}$ as Matveev's ${b_n, b_1,\ldots, b_{n-1}}$, respectively,
${\Theta_0, \Theta_1, \ldots, \Theta_r}$ as Matveev's ${A_n,A_1, \ldots, A_{n-1}}$, respectively, and $B$ as Matveev's~$B$.

Notice that Matveev assumes (in our notations) that
\begin{equation}
\label{ematlog}
\Theta_k \ge |\log \theta_k|,
\end{equation}
with some choice of the complex value of the logarithm. However, if we pick the principal value of the logarithm, then
$$
|\log \theta_k| \le |\log |\theta_k||+\pi \le d\h(\theta_k)+\pi \le (1+\pi) \Theta_k.
$$
Hence we may disregard~\eqref{ematlog} at the cost of increasing the absolute constant~$C$ in the definition of~$\Upsilon$.

In the case of non-archimedean~$v$ we employ the result of Yu~\cite{Yu07}. Precisely, we use the second consequence of his ``Main Theorem'' on page~190 (see the bottom of page~190 and the top of page~191), which asserts that,  assuming~(1.19) of~\cite{Yu07}, but without assuming~(1.5) and~(1.15), the first displayed equation on the top of page~191 of~\cite{Yu07} holds.

In our notations, taking, as in the archimedean case, ${n=r+1}$, using ${1, b_1, \ldots, b_r}$ as Yu's ${b_n, b_1,\ldots, b_{n-1}}$, noticing that Yu's parameters ${h_n,h_1,\ldots, h_{n-1}}$ do not exceed our ${d^{-1}\Theta_0, d^{-1}\Theta_1,\ldots, d^{-1}\Theta_r}$, and setting Yu's~$B_n$ to be~$1$, we re-state Yu's result as follows.  Let $\mathfrak{p}$ be the prime ideal corresponding to $v$ and~$\delta$ a real number satisfying ${0<\delta\le1/2}$; then
\begin{align*}
&\ord_\mathfrak{p}(\Lambda-1) < (Cd)^{2r+5} \frac{p_{v}^d}{(\log p_{v})^2}\max \left\{\Theta_0\Theta_1\cdots\Theta_r\log Q,  \delta B\right\},\\
&Q= \delta^{-1}e^{6r^2}d^{2r}p_{v}^{rd}\Theta_1\cdots\Theta_r.
\end{align*}
Here we replace Yu's~$c_0$ by $d^{r+1}$, Yu's~$c_1$ by $e^{6r^2}d^{3r}$ and Yu's~$C_0$ by ${(Cd)^{3r+6}p_{v}^d(\log p_{v})^{-2}}$,   the constant~$C$ being absolute.
Observing that
$$
\log Q= \log \left(\delta^{-1}\Theta_1\cdots\Theta_r\right)+ O(r^2d\log p_{v}),
$$
and modifying the absolute constant~$C$, we obtain
\begin{equation}
\label{eyu}
\ord_\mathfrak{p}(\Lambda-1) < (Cd)^{2r+6} \frac{p_{v}^d}{\log p_{v}}\max \left\{\Theta_0\Theta_1\cdots\Theta_r\log \left(\delta^{-1}\Theta_1\cdots\Theta_r\right) , \delta B\right\}.
\end{equation}
Notice that $B\ge 3$, then $\log B> 1$. Set now
$$
\delta = \min\left\{\Theta_1\cdots\Theta_r\frac{\log B}{B}, \frac12\right\}.
$$
If ${\delta<1/2}$ then the maximum in~\eqref{eyu} doesn't exceed $\Theta_0\Theta_1\cdots\Theta_r\log B$. And if ${\delta=1/2}$, then
$$
\frac{B}{\log B} \le 2 \Theta_1\cdots\Theta_r,
$$
which, by \cite[Lemma 2.3.3]{BH1}, implies that
$$
B \le 4\Theta_1\cdots\Theta_r\log\left(2\Theta_1\cdots\Theta_r\right) \le 4(r+1) \Theta_1\cdots\Theta_r\log B,
$$
and the maximum in~\eqref{eyu} is at most $2(r+1)\Theta_0\Theta_1\cdots\Theta_r\log B$. So in any case we obtain (again slightly adjusting the absolute constant~$C$)  the estimate
\begin{equation}
\ord_\mathfrak{p}(\Lambda-1) < (Cd)^{2r+6} \frac{p_{v}^d}{\log p_{v}}\Theta_0\Theta_1\cdots\Theta_r \log B.
\end{equation}
Finally, since
${|\Lambda-1|_v= e^{- \frac{\log p_{v}}{e_\mathfrak{p}}\ord_\mathfrak{p}(\Lambda-1)}}$, where~$e_\mathfrak{p}$ is the absolute ramification index of~$\mathfrak{p}$, we obtain the result in the non-archimedean case as well.
\end{proof}

\begin{remark}
{\rm
We choose the form of Baker's inequality in Theorem \ref{tbaker} because of its convenience for our computations, although it is effective but not explicit. If one want to get an explicit bound for $\h(P)$, he can apply Matveev \cite{Matveev} and Yu \cite{Yu07} respectively, like \cite{GY},
and he also can apply \cite[Theorem C]{Gyory} to handle uniformly with the Archimedean and non-Archimedean cases.
}
\end{remark}

\section{The case of mixed level}
\label{mixed level}
In this section, we assume that $N$ has at least two distinct prime factors. Then we will apply Baker's inequality to prove Theorem \ref{main1}
and \ref{main2}.

In the sequel, we assume that $P$ is an $S_{0}$-integral point of $X_{G}$ and $\nu_{\infty}(G)\ge 3$.
What we want to do is to obtain some bounds for $\h(P)$.

From now on we let $K=K_{0}\cdot\Q(\zeta_{N})=K_{0}(\zeta_{N})$.
Let $S$ be the set consisting of the extensions of the places from $S_{0}$ to $K$, i.e.
$$
S=\{v\in M_{K}:v|v_{0}\in S_{0}\}.
$$
Then $P$ is also an $S$-integral point of $X_{G}$.

Put $d=[K:\Q]$, $s=|S|$ and $r=s-1$. Since $j(P)\in \OO_{S}$, we have
$$
\h(P)=d^{-1}\sum\limits_{v\in S}d_{v}\log^{+}|j(P)|_{v}\le\sum\limits_{v\in S}\log^{+}|j(P)|_{v}.
$$
 Then there exists some $w\in S$ such that
$$
\h(P)\le s\log|j(P)|_{w}.
$$
We fix this valuation $w$ from now on. Therefore, we only need to bound $\log|j(P)|_{w}$.

As the discussion in Subsection \ref{X_{G_{1}}}, $P$ is also an $S$-integral point of $X_{G_{1}}$.
Hence for our purposes, we only need to focus on the modular curve $X_{G_{1}}$.

We partition the set $S$ into three pairwise disjoint subsets: $S=S_{1}\cup S_{2}\cup S_{3}$,
where $S_{1}$ consists of places $v\in S$ such that $P\in X_{G_{1}}(\bar{K}_{v})^{+}$,
$S_{2}=M_{K}^{\infty}\setminus S_{1}$, and $S_{3}=S\setminus (S_{1}\cup S_{2})$.

From now on, for $v\in S_{1}$ let $c_{v}$ be a $v$-nearby cusp of $P$, and we write $q_{v}$ for $q_{c_{v}}$ and $e_{v}$ for $e_{c_v}$.
Notice that for any $v\in S_{3}$, it is non-Archimedean with $|j(P)|_{v}\le 1$.

In the sequel we can assume that $|j(P)|_{w}>3500$, otherwise we can get a
better bound than those given in Section 1. Then we have $w\in S_{1}$ and $P\in \Omega_{c_w,w}$.
Therefore, by (\ref{nearby2}) we only need to bound $\log |q_{w}(P)^{-1}|_{w}$.

From now on we assume that $|q_{w}(P)|_{w}\le 10^{-N}$. Indeed, applying (\ref{nearby2}) the inequality $|q_{w}(P)|_{w}> 10^{-N}$ yields 
$\h(P)<3sN$, which is a much better estimate for $\h(P)$ than those given in Section 1.

Notice that under our assumptions, we see that $N\ge 2$. Moreover, in this section we assume that $s\ge 2$.
In fact, if $s=1$, then we can add another valuation to $S$ such that $s=2$, and then the final results of this section also hold.

\subsection{Preparation for Baker's inequality}\label{preparation}
We fix  an orbit $\OO$ of the group action $\mathcal{A}_{N}/G_{1}$ as follows.
Put $U=u_{\OO}$, where $u_{\OO}$ is defined in (\ref{uo}).

If $\ord_{c_w}U\ne 0$, we choose $\OO$ such that $\ord_{c_w}U< 0$ according to Proposition \ref{u1}.
Noticing $v_{\infty}(G)\ge 3$ and combining with Proposition \ref{munit} (vi), we can choose another orbit $\OO^{\prime}$
such that $U$ and $V$ are multiplicatively independent modulo constants with $\ord_{c_w}V>0$, where $V=u_{\OO^{\prime}}$.

Define the following function
\begin{equation}
W=\left\{ \begin{array}{ll}
                 U & \textrm{if $\ord_{c_w}U=0$},\\
                  \\
                 U^{\ord_{c_w}V}V^{-\ord_{c_w}U} & \textrm{if $\ord_{c_w}U\ne 0$}.
                 \end{array} \right.
\notag
\end{equation}
So we always have $\ord_{c_w}W=0$ and $W(P)\in \OO_{S}$.
In particular, $W$ is integral over $\Z[j]$. Moreover,  $W$ is not a constant by Proposition \ref{munit} (vi).

By Proposition \ref{munit} (ii) and (iii), we have
\begin{equation}
\gamma_{w}^{-1}W(P)=1+O_{w}(4^{24N^{7}}|q_{w}(P)|_{w}^{1/N}),
\end{equation}
where
\begin{equation}
\gamma_w=\left\{ \begin{array}{ll}
                 \gamma_{\OO,c_w} & \textrm{if $\ord_{c_w}U=0$},\\
                  \\
                 \gamma_{\OO,c_w}^{\ord_{c_w}V}\gamma_{\OO^{\prime},c_w}^{-\ord_{c_w}U} & \textrm{if $\ord_{c_w}U\ne 0$};
                 \end{array} \right.
\notag
\end{equation}
and
$$
\h(\gamma_w)\le 24N^{7}\log 2.
$$

By Proposition \ref{u1}, we know that $W(P)$ is a unit of $\OO_{S}$. So there exist some integers $b_{1},\cdots, b_{r}\in\Z$ such that
$W(P)=\omega\eta_{1}^{b_{1}}\cdots\eta_{r}^{b_{r}}$, where $\omega$ is a root of unity and $\eta_{1},\cdots,\eta_{r}$ are from Proposition \ref{S-unit}.
Let $\eta_{0}=\omega\gamma_{w}^{-1}$. Then we set
\begin{equation}
\Lambda=\gamma_{w}^{-1}W(P)=\eta_{0}\eta_{1}^{b_{1}}\cdots\eta_{r}^{b_{r}}.
\end{equation}
Notice that $\eta_{0},\cdots,\eta_{r}\in K$ and
\begin{equation}\label{Lambda1}
|\Lambda-1|_{w}\le 4^{24N^{7}}|q_{w}(P)|_{w}^{1/N}.
\end{equation}

For subsequent deductions, we need to bound $\h(W(P))$.
\begin{proposition}\label{h(W)}
We have
$$
\h(W(P))\le 2sN^{8}\log|q_w^{-1}(P)|_w+94sN^{8}\log N. 
$$
\end{proposition}
\begin{proof}
First suppose that $\ord_{c_w}U=0$. Then $W=U$.
For $v\in S_{3}$, $j(P)$ is a $v$-adic integer.
Hence, so is the number $W(P)$.
In addition, it is easy to see that
$$
\sum\limits_{v\in M_{K}^{\infty}}d_{v}\rho_{v}=12dN^{3}\log N, \quad \sum\limits_{v\in M_{K}^{0}}d_{v}\rho_{v}\le 12dN^{3}\log N.
$$

Notice that for $v\in S_{1}$, $|\ord_{c_v}(W)|\le N^{4}$. Applying Proposition \ref{munit} (iv) and (\ref{nearby2}),  we have
\begin{align*}
d^{-1}\sum\limits_{v\in S_{1}}d_{v}\log^{+}|W(P)|_{v}
&\le N^{4}d^{-1}\sum\limits_{v\in S_{1}}d_{v}\log|q_v(P)^{-1}|_v+d^{-1}\sum\limits_{v\in S_{1}}d_{v}\rho_{v}\\
&\le N^{4}d^{-1}\sum\limits_{v\in S_{1}}d_{v}\log|j(P)|_v+sN^{4}\log\frac{3}{2}+24N^{3}\log N\\
&\le N^{4}\h(P)+sN^{4}\log\frac{3}{2}+24N^{3}\log N\\
&\le sN^{4}\log|j(P)|_w+sN^{4}\log\frac{3}{2}+24N^{3}\log N\\
&\le sN^{4}\log|q_w(P)^{-1}|_w+sN^{4}\log 3+24N^{3}\log N.
\end{align*}

It follows from Proposition \ref{munit} (v) that
$$
d^{-1}\sum\limits_{v\in S_{2}}d_{v}\log^{+}|W(P)|_{v}
\le N^{3}\log 5900 +12N^{3}\log N.
$$

Hence, we get
\begin{align*}
\h(W(P))&=d^{-1}\sum\limits_{v\in S_{1}\cup S_{2}}d_{v}\log^{+}|W(P)|_{v}\\
&\le sN^{4}\log|q_w(P)^{-1}|_w+sN^{4}\log 3+36N^{3}\log N+N^{3}\log 5900.
\end{align*}

Now suppose that $\ord_{c_w}U\ne 0$. For any $v\in S_{1}$, we have
$$
|\log|W(P)|_{v}|\le\frac{|\ord_{c_v}(W)|}{e_{v}}\log|q_{v}(P)^{-1}|_{v}+ 2N^{4}\rho_{v}.
$$
Here note that $|\ord_{c_v}(W)|\le 2N^{8}$.
For any $v\in M_{K}^{\infty}$, we have
$$
|\log|W(P)|_{v}|\le 2N^{7}\log(|j(P)|_{v}+2400)+2N^{4}\rho_{v}.
$$
Apply the same argument as the above, we obtain
$$
\h(W(P))\le 2sN^{8}\log|q_w(P)^{-1}|_w+2sN^{8}\log 3+72N^{7}\log N+2N^{7}\log 5900.
$$

Now it is easy to get the desired result.
\end{proof}

\subsection{Using Baker's inequality}

If $\Lambda=1$, we can get better bounds for $\h(P)$ than those given in Section 1, see Section \ref{specialcase}.
So in the rest of this section we assume that $\Lambda\ne 1$.

Let $B^{*}=\max\{|b_{1}|,\cdots,|b_{r}|\}$, and let $\Theta_{0},\Theta_{1},\cdots,\Theta_{r}$ be real numbers
satisfying
\begin{align*}
&\Theta_{k}\ge \max\{d\h(\eta_{k}),1\}, \quad k=0,\cdots,r.
\end{align*}

By Theorem \ref{tbaker}, there exists an absolute constant $C$ which can be
determined explicitly such that the following holds. Choosing $B\ge B^{*}$ and $B\ge \max\{3, \Theta_1,\cdots,\Theta_r\}$, we have
\begin{equation}\label{Lambda2}
|\Lambda-1|_{w}\ge e^{-\Upsilon\Theta_{0}\Theta_{1}\cdots\Theta_{r}\log B},
\end{equation}
where
$$
\Upsilon =
\begin{cases}
C^r d^2\log(2d), & w\mid \infty,\\
(Cd)^{2r+6}p^d, & \textrm{otherwise}.
\end{cases}
$$
Recall that $p$ has been defined in Section 1.

Applying (\ref{Lambda1}), we have
$$
e^{-\Upsilon\Theta_{0}\Theta_{1}\cdots\Theta_{r}\log B}\le 4^{24N^{7}} |q_{w}(P)|_{w}^{1/N}.
$$
Hence, we obtain
\begin{equation}\label{q1}
\log |q_{w}(P)^{-1}|_{w}\le N\Upsilon\Theta_{0}\Theta_{1}\cdots\Theta_{r}\log B
+48N^{8}\log 2.
\end{equation}

According to Proposition \ref{S-unit}, we can choose
\begin{align*}
&\Theta_{k}=d\zeta\h(\eta_{k}),\quad k=1,\cdots,r.
\end{align*}
So we have
$$
\Theta_{1}\cdots\Theta_{r}\le r^{2r}\zeta^{r}R(S).
$$

Since
\begin{align*}
d\h(\eta_{0})= d\h(\gamma_w)\le 24dN^{7}\log 2,
\end{align*}
we can choose
$$
\Theta_{0}=24dN^{7}\log 2.
$$

Corollary \ref{eta} tells us that
\begin{align*}
B^{*}
&\le 2dr^{2r}\zeta\h(W(P)).
\end{align*}
Notice that we also need $B\ge \max\{3, \Theta_1,\cdots,\Theta_r\}$, by Proposition \ref{S-unit} and Proposition \ref{h(W)} we can choose
$$
B=r^{2r}\zeta^{r}R(S)+2dr^{2r}\zeta\left(2sN^{8}\log|q_w(P)^{-1}|_w+94sN^{8}\log N\right).
$$
Again, we write $B=\alpha\log|q_w(P)^{-1}|_w+\beta$, where
\begin{align*}
&\alpha=4dsr^{2r}\zeta N^{8},\\
&\beta=r^{2r}\zeta^{r}R(S)+188dsr^{2r}\zeta N^{8}\log N.
\end{align*}

Hence, (\ref{q1}) yields
$$
\alpha\log|q_w(P)^{-1}|_w+\beta \le \alpha N\Upsilon\Theta_{0}\Theta_{1}\cdots\Theta_{r}\log(\alpha\log|q_w(P)^{-1}|_w+\beta)
+48\alpha N^{8}\log 2+\beta.
$$
Here we put $C_{1}=\alpha N\Upsilon\Theta_{0}\Theta_{1}\cdots\Theta_{r}$ and $C_{2}=48\alpha N^{8}\log 2+\beta$, then
$$
\alpha\log|q_w(P)^{-1}|_w+\beta \le C_{1}\log(\alpha\log|q_w(P)^{-1}|_w+\beta)+C_{2}.
$$
Therefore, by \cite[Lemma 2.3.3]{BH1} we obtain
$$
\alpha\log|q_w(P)^{-1}|_w+\beta\le 2(C_{1}\log C_{1}+C_{2}).
$$
Hence
$$
\log|q_w(P)^{-1}|_w\le 2\alpha^{-1}C_{1}\log C_{1}+\alpha^{-1}(2C_{2}-\beta).
$$
That is
$$
\log|j(P)|_{w}\le 2\alpha^{-1}C_{1}\log C_{1}+\alpha^{-1}(2C_{2}-\beta)+\log 2.
$$

So we have
\begin{equation}
\h(P)\le 2s\alpha^{-1}C_{1}\log C_{1}+s\alpha^{-1}(2C_{2}-\beta)+s\log 2.
\notag
\end{equation}

Finally we get
\begin{equation}\label{h(P)}
\h(P)\ll dsr^{2r}\zeta^{r}N^{8}\Upsilon R(S)\log(d^{2}sr^{4r}\zeta^{r+1}N^{16}\Upsilon R(S)).
\end{equation}

To get a bound for $\h(P)$, we only need to calculate the quantities in the above inequality.

\subsection{Proof of Theorem \ref{main1}}

Under the assumptions of Theorem \ref{main1}, we have  $K=\Q(\zeta_{N})$ and $S=M_{K}^{\infty}$.
Since we have assumed that $s\ge 2$, we have $\varphi(N)\ge 4$.

Then $|D|\le N^{\varphi(N)}$ according to \cite[Proposition 2.7]{Wa}. It follows from Proposition \ref{S-regulator} that
$$
R(S)\ll \varphi(N)^{-1}N^{\varphi(N)/2}(\log N)^{\varphi(N)-1}.
$$

Notice that
\begin{align*}
&s=\varphi(N)/2,\\
&\zeta \ll (\log \varphi(N))^{3},\\
&\Upsilon=C^{\frac{\varphi(N)}{2}-1}\varphi(N)^2\log (2\varphi(N)),\\
&\log(d^{2}sr^{4r}\zeta^{r+1}N^{16}\Upsilon R(S))\ll \varphi(N)\log N.
\end{align*}
Applying (\ref{h(P)}) we obtain
\begin{align*}
\h(P)
&\le C^{\varphi(N)}(\varphi(N))^{\varphi(N)+2}(\log \varphi(N))^{\frac{3}{2}\varphi(N)-2}N^{\frac{1}{2}\varphi(N)+8}(\log N)^{\varphi(N)},\\
&\le C^{\varphi(N)}N^{\frac{3}{2}\varphi(N)+10}(\log N)^{\frac{5}{2}\varphi(N)-2},
\end{align*}
the constant $C$ being modified. Hence we prove Theorem \ref{main1}.

\subsection{Proof of Theorem \ref{main2}}
Now we need to give a bound for $\h(P)$ based on the parameters of $K_{0}$ with the assumptions of Theorem \ref{main2}.

Firstly, notice that
\begin{align*}
&s\le s_{0}\varphi(N),\\
&r=s-1\le s_{0}\varphi(N)-1,\\
&d\le d_{0}\varphi(N),\\
&\zeta \ll (\log d)^{3}\le (\log (d_{0}\varphi(N)))^{3}.
\end{align*}

Using Proposition \ref{S-regulator}, we estimate $R(S)$ as follows:
$$
R(S)\ll d^{-d}\sqrt{|D|}(\log |D|)^{d-1}\prod\limits_{\substack{v\in S\\v\nmid \infty}}\log\NN_{K/\Q}(v).
$$
Since $\NN_{K/\Q}(v)\le p^{[K:\Q]}=p^{d}$, this implies the upper bound
\begin{equation}\label{logR}
\log R(S)\ll \frac{1}{2}\log |D|+d\log\log |D|+s\log (dp).
\end{equation}

Let $D_{K/K_{0}}$ be the relative discriminant of $K/K_{0}$. We have
$$
D=\mathcal{N}_{K_{0}/\Q}(D_{K/K_{0}})D_{0}^{[K:K_{0}]}.
$$
We denote by $\OO_{K_{0}}$ and $\OO_{K}$ the ring of integers of $K_0$ and $K$ respectively. Since $K=K_{0}(\zeta_{N})$, we have
$$
\OO_{K_{0}}\subseteq \OO_{K_{0}}[\zeta_{N}]\subseteq \OO_{K}.
$$
By \cite[III (2.20) (b)]{Frohlich} and note that the absolute value of the discriminant of the polynomial $x^{N}-1$ is $N^{N}$, we get
$$
D_{K/K_{0}}|N^{N}.
$$
So
$$
|\mathcal{N}_{K_{0}/\Q}(D_{K/K_{0}})|\le N^{d_{0}N}.
$$
Hence
$$
|D|\le N^{d_{0}N}|D_{0}|^{\varphi(N)}.
$$

Now let $v_{0}$ be a non-archimedean place of $K_{0}$, and $v_{1},\cdots,v_{m}$ all its
extensions to $K$, their residue degrees over $K_{0}$ being $f_{1},\cdots,f_{m}$ respectively.
Then $f_{1}+\cdots+f_{m}\le [K:K_{0}]\le \varphi(N)$, which implies that $f_{1}\cdots f_{m}\le 2^{\varphi(N)}$.
Notice that we always have $2\log\NN_{K_{0}/\Q}(v_{0})>1$.
Since $\NN_{K/\Q}(v_{k})=\NN_{K_{0}/\Q}(v_{0})^{f_{k}}$ for $1\le k\le m$ and $m\le \varphi(N)$, we have
\begin{align*}
\prod\limits_{k=1}^{m}\log\NN_{K/\Q}(v_k)
&\le 2^{\varphi(N)}(\log\NN_{K_{0}/\Q}(v_{0}))^m\\
&\le 2^{\varphi(N)}(2\log\NN_{K_{0}/\Q}(v_{0}))^m\\
&\le 4^{\varphi(N)}(\log\NN_{K_{0}/\Q}(v_{0}))^{\varphi(N)}.
\end{align*}
Hence
\begin{equation}\label{logN}
\prod\limits_{\substack{v\in S\\v\nmid \infty}}\log\NN_{K/\Q}(v)
\le 4^{s_{0}\varphi(N)}\left(\prod\limits_{\substack{v\in S_{0}\\v\nmid \infty}}\log\NN_{K_{0}/\Q}(v)\right)^{\varphi(N)}.
\end{equation}

If we now denote by $\Delta$ the quantity defined in \eqref{Delta}, then using \eqref{logR} and \eqref{logN},
we obtain the following estimates:
\begin{align*}
& R(S)\ll 4^{s_{0}\varphi(N)}\Delta,\\
& R(S)\log R(S)\ll 4^{s_{0}\varphi(N)}s_{0}\Delta\log p,\\
& R(S)\log(d^{2}sr^{4r}\zeta^{r+1}N^{16}\Upsilon R(S))\ll 4^{s_{0}\varphi(N)}s_{0}\Delta\log (ps_{0}).
\end{align*}
Here we always choose $\Upsilon=(Cd)^{2r+6}p^d$.

Finally, using (\ref{h(P)}) and noticing that $d_0\le 2s_0$, we get
\begin{align*}
\h(P)&\le \left(Cd_{0}s_{0}\varphi(N)^{2}\right)^{2s_{0}\varphi(N)}(\log(d_{0}\varphi(N)))^{3s_{0}\varphi(N)}N^{8}p^{d_{0}\varphi(N)}\Delta\log p\\
&\le \left(Cd_{0}s_{0}N^{2}\right)^{2s_{0}N}(\log(d_{0}N))^{3s_{0}N}p^{d_{0}N}\Delta.
\end{align*}
the constant $C$ being modified.

Therefore, Theorem \ref{main2} is proved.

\section{The case of prime power level}\label{prime power}
In this section, we assume that $N$ is a prime power.

As Section \ref{mixed level}, we can define a similar function $W$. But in this case $W(P)$ is not a unit of $\OO_{S}$
by Proposition \ref{u1}.
So we need to raise the level. Put
\begin{equation}
M=\left\{ \begin{array}{ll}
                 2N & \textrm{if $N$ is not a power of 2},\\
                  \\
                 3N & \textrm{if $N$ is a power of 2}.
                 \end{array} \right.
\notag
\end{equation}
Notice that $X_{G}$ is also a modular curve of level $M$ and $\nu_{\infty}(G)\ge 3$, since we have the following natural sequence of morphisms
$$
X(M)\to X(N)\to X_{G}\to X(1).
$$

Since $\Gal(\Q(X(M))/\Q(j))=\GL_{2}(\Z/M\Z)/\pm 1$, $X_{G}$ corresponds to a subgroup $\widetilde{G}$ of $\GL_{2}(\Z/M\Z)$ containing $\pm 1$.
In fact, The restriction of $\widetilde{G}$ on $X(N)$ is $G$. The modular curve $X_{\widetilde{G}}$ has the same integral geometric model as $X_{G}$.
In particular, $P$ is also an $S_{0}$-integral point of $X_{\widetilde{G}}$.

Therefore, from Theorem \ref{main1} and \ref{main2}, we can get two upper bounds for $\h(P)$ by replacing $N$ by $M$, which proves Theorem \ref{main3}.

\section{The case $\Lambda =1$}
\label{specialcase}
In this section, we suppose that $N$ is not a prime power without loss of generality.  Under the assumption $\Lambda=1$ we can obtain better bounds for $\h(P)$ than those given in Section 1.

Let $c$ be a cusp of $X_{G_1}$ and $v\in M_{K}$. We also denote by $v$ the unique extension of $v$ to $\bar{K}_{v}$.
Recall $\Omega_{c,v}$ and the $q$-parameter $q_c$ mentioned in Section \ref{X_{G_{1}}}, for the modular function $U$ defined in Section \ref{preparation}, we get the following lemma.
\begin{lemma}\label{Taylor1}
There exist an integer-valued function $f(\cdot)$ with respect to $q_{c}$ and $\lambda_{1}^{c}, \lambda_{2}^{c}, \lambda_{3}^{c}\cdots\in \Q(\zeta_{N})$ such that the following identity holds in $v$-adic sense,
\begin{equation}\label{Taylor0}
\log\frac{U(q_c)}{\gamma_{\OO,c}q_{c}^{\frac{\ord_{c}U}{e_{c}}}}=2\pi f(q_{c})i+\sum\limits_{k=1}^{\infty}\lambda_{k}^{c}q_{c}^{k/N},
\end{equation}
and
\begin{equation}
|\lambda_{k}^{c}|_{v}\le\left\{ \begin{array}{ll}
                 |k|_{v}^{-1} & \textrm{if $v$ is finite},\\
                 24N^{2}(k+N) & \textrm{if $v$ is infinite}.
                 \end{array} \right.
\notag
\end{equation}
In particular, for every $k\ge 1$ we have
$$
\h(\lambda_{k}^{c})\le \log(24N^{3}+24kN^{2})+\log k.
$$
\end{lemma}
\begin{proof}
By definition, we have
\begin{equation}\label{U(q_c)}
\frac{U(q_c)}{\gamma_{\OO,c}q_c^{\frac{\ord_{c}U}{e_{c}}}}=
\prod\limits_{{\rm\bf a}\in \OO}\prod\limits_{\substack{n=0\\n+a_{1}\ne 0}}^{\infty}(1-q_{c}^{n+a_{1}}e^{2\pi ia_{2}})^{12N}
\prod\limits_{n=0}^{\infty}(1-q_{c}^{n+1-a_{1}}e^{-2\pi ia_{2}})^{12N}.
\end{equation}
Since
$$
\sum\limits_{{\rm\bf a}\in \OO}\left(\sum\limits_{\substack{n=0\\n+a_{1}\ne 0}}^{\infty}12N|q_{c}|_{v}^{n+a_{1}}+
\sum\limits_{n=0}^{\infty}12N|q_{c}|_{v}^{n+1-a_{1}}\right)
$$
is convergent, it follows from \cite[Chapter 5 Section 2.2 Theorem 6]{Ahlfors} that the right-hand side of (\ref{U(q_c)}) is absolutely convergent ($v$ is infinite). It is also true when $v$ is finite.
Then we can write (\ref{U(q_c)}) as the form $\prod\limits_{n=1}^{\infty}(1+d_{n})$ such that $\prod\limits_{n=1}^{\infty}(1+d_{n})$ is absolutely convergent.
Hence, \cite[Chapter 5 Section 2.2 Theorem 5]{Ahlfors} ($v$ is infinite) and \cite[Chapter IV Section 2]{Koblitz} ($v$ is finite) give
\begin{align*}
&\log \frac{U(q_c)}{\gamma_{\OO,c}q_c^{\frac{\ord_{c}U}{e_{c}}}}\\
&=2\pi f(q_{c})i+
\sum\limits_{{\rm\bf a}\in \OO}\left(\sum\limits_{\substack{n=0\\n+a_{1}\ne 0}}^{\infty}12N\log(1-q_{c}^{n+a_{1}}e^{2\pi ia_{2}})+
\sum\limits_{n=0}^{\infty}12N\log(1-q_{c}^{n+1-a_{1}}e^{-2\pi ia_{2}})\right),
\end{align*}
where by default $f(q_{c})$ is always equal to 0 if $v$ is finite.
Applying the Taylor expansion of the logarithm function to the right-hand side of the above formula, we get the desired formula for
$\log\frac{U(q_c)}{\gamma_{\OO,c}q_{c}^{\frac{\ord_{c}U}{e_{c}}}}$.

For a fixed non-negative integer $n$ (where we assume $n>0$ if $a_{1}=0$), write
\begin{equation}\label{Taylor}
\log(1-q_{c}^{n+a_{1}}e^{2\pi ia_{2}})=\sum\limits_{k=1}^{\infty}\alpha_{k}q^{k/N}.
\notag
\end{equation}
An immediate verification shows that
\begin{equation}
|\alpha_{k}|_{v}\le\left\{ \begin{array}{ll}
                 |k|_{v}^{-1} & \textrm{if $v$ is finite},\\
                 1 & \textrm{if $v$ is infinite}.
                 \end{array} \right.
\notag
\end{equation}
Same estimates hold true for the coefficients of the $q$-series for $\log(1-q_{c}^{n+1-a_{1}}e^{-2\pi ia_{2}})$.

For each ${\rm\bf a}\in \OO$, the number of coefficients in the $q$-series for $\log(1-q_{c}^{n+a_{1}}e^{2\pi ia_{2}})$ which may contribute to $\lambda_{k}^{c}$ (those with $0\le n\le k/N$) is at most $k/N+1$, and the same is true for the $q$-series for $\log(1-q_{c}^{n+1-a_{1}}e^{-2\pi ia_{2}})$. The bound for $|\lambda_{k}^{c}|_{v}$ now follows by summation.
\end{proof}

\begin{corollary}\label{Taylor2}
Assume that $\ord_{c}U=0$. Then $\lambda_{k}^{c}\ne 0$ for some $k\le N^{6}$.
\end{corollary}
\begin{proof}
Since $U$ is not a constant, there must exist some $\lambda_{k}^{c}\ne 0$. Under the assumption $\ord_{c}U=0$,
we have $U(c)=\gamma_{\OO,c}$, and then $f(q_c(c))=0$ by (\ref{Taylor0}). We extend the additive valuation $\ord_{c}$ from the field $K(X_{G_1})$ to the field of formal
power series $K((q_{c}^{1/e_c}))$. Then $\ord_{c}q_{c}^{1/e_c}=1$ and
$\ord_{c}\left(-2\pi f(q_{c})i+\log (U/\gamma_{\OO,c})\right)\le\ord_{c}\log (U/\gamma_{\OO,c})=\ord_{c}(U/\gamma_{\OO,c}-1)$. The latter quantity is bounded by the degree of
$U/\gamma_{\OO,c}-1$, which is equal to the degree of $U$.

The degree of $U$ is equal to $\frac{1}{2}\sum\limits_{c_{0}}\left|\ord_{c_{0}}\, U\right|$,
here the sum runs through all the cusps of $X_{G_{1}}$. Then the result follows from Proposition \ref{munit} (ii).
\end{proof}

Now we can prove a general result.
\begin{proposition}
Assume that $\ord_{c}U=0$. Then for $P\in \Omega_{c,v}$ such that $U(P)=\gamma_{\OO,c}$, we have
$$
\log|q_c(P)^{-1}|_{v}\le N\varphi(N)\log(24N^{14}+24N^{9})+N\log(48N^{2}(N^{6}+N+1)).
$$
\end{proposition}
\begin{proof}
Let $n$ be the smallest $k$ such that $\lambda_{k}^{c}\ne 0$. Then $n\le N^{6}$.
We assume that $|q_c(P)|_{v}\le 10^{-N}$, otherwise there is nothing to prove. Since $\ord_{c}U=0$
and $U(P)=\gamma_{\OO,c}$, it follows from Lemma \ref{Taylor1} that $2\pi f(q_{c}(P))i+\sum\limits_{k=n}^{\infty}\lambda_{k}^{c}q_{c}(P)^{k/N}=0$.

Suppose that $f(q_{c}(P))=0$. Then $|\lambda_{n}^{c}q_{c}(P)^{n/N}|_{v}=|\sum\limits_{k=n+1}^{\infty}\lambda_{k}^{c}q_{c}(P)^{k/N}|_{v}$.
 On one side, we have
\begin{align*}
|\sum\limits_{k=n+1}^{\infty}\lambda_{k}^{c}q_{c}(P)^{k/N}|_{v}&\le \sum\limits_{k=n+1}^{\infty}|\lambda_{k}^{c}|_{v}|q_{c}(P)|_{v}^{k/N}\\
&\le \sum\limits_{k=n+1}^{\infty}24N^{2}(k+N)|q_{c}(P)|_{v}^{k/N}\\
&=48N^{2}(n+N+1)|q_{c}(P)|_{v}^{(n+1)/N}.
\end{align*}
On the other side, using Liouville's inequality (see \cite[Formula (3.13)]{Waldschmidt2000}), we get
$$
|\lambda_{n}^{c}|_{v}\ge e^{-[\Q(\zeta_{N}):\Q]\h({\lambda_{n}^{c}})}\ge (24nN^{3}+24n^{2}N^{2})^{-\varphi(N)}.
$$
Then the desired result follows easily.

Suppose that $f(q_{c}(P))\ne 0$. Then $2\pi\le|\sum\limits_{k=n}^{\infty}\lambda_{k}^{c}q_{c}(P)^{k/N}|_{v}\le 48N^{2}(n+N)|q_{c}(P)|_{v}^{n/N}$.
Then we get
$$
\log|q_c(P)^{-1}|_{v}\le N\log(48N^{2}(N^{6}+N)).
$$
\end{proof}

Now we assume that $\ord_{c_w}U=0$. Then we have $W=U$. Since $\Lambda =1$, $W(P)=\gamma_{\OO,c_w}$.
For the $S$-integral point $P$ of $X_{G_1}$ fixed in Section \ref{mixed level}, applying the above proposition to $W$, we obtain
\begin{align*}
\h(P)&\le s(\log|q_w(P)^{-1}|_{w}+\log 2)\\
&\le s_{0}N\left(N\varphi(N)\log(24N^{14}+24N^{9})+N\log(48N^{2}(N^{6}+N+1))+\log 2\right).
\end{align*}

Now we assume that $\ord_{c_w}U\ne 0$. Then $W=U^{\ord_{c_w}V}V^{-\ord_{c_w}U}$ with $\ord_{c_w}W=0$.
Proposition \ref{munit} (vi) guarantees that $W$ is not a constant.
Applying the same method as the above without difficulties, we can also get a better bound than Theorem \ref{main1} and \ref{main2}.
We omit the details here.

In conclusion, if assuming $\Lambda=1$, we can get polynomial bounds for $\h(P)$ in terms of $s_{0}$ and $N$, which are obviously
better than those in Theorems \ref{main1}-\ref{main3}.

\section*{Acknowledgement}
The author would like to thank his advisor Yuri Bilu for lots of stimulating suggestions, helpful discussions and careful reading, especially for his key suggestion in Section \ref{specialcase}. He also thank
Aur\'elien Bajolet for valuable discussions.
He is also grateful to the referee for careful reading and very useful comments.


\begin{thebibliography}{99}

\bibitem{ABBN}
M. Abouzaid, A. B\' erczes, Yu. Bilu and S. Najib, \emph{Effective bounds for the polynomial norm equation}, in preparation.

\bibitem{Ahlfors}
L. Ahlfors, \textit{Complex Analysis}, Third edition, Mcgraw-Hill, 1979.

\bibitem{Apostol}
T.M. Apostol, \textit{Introduction to Analytic Number Theory}, Undergraduate Texts in
Mathematics, Springer, New York, 1976.

\bibitem{BaBi11}
A. Bajolet and Yu. Bilu, \emph{Computing integral points on $X_{\ns}^{+}(p)$}, preprint, 2012. \textsf{arXiv:1212.0665}

\bibitem{BaSh1}
A. Bajolet and M. Sha, \emph{Bounding the $j$-invariant of integral points on $X_{\ns}^{+}(p)$}, Proc. Amer. Math. Soc., to appear; \textsf{arXiv:1203.1187v2}.

\bibitem{Gyory}
A. B\' erczes, J. Evertse and K. Gy\H ory,  \emph{Effective results for linear equations in two unknowns from a multiplicative division group},
Acta Arithm. \textbf{136} (2009), 331-349.

\bibitem{Bi95}
Yu. Bilu, \emph{Effective analysis of integral points on algebraic curves}, Israel J. Math. \textbf{90} (1995), 235-252.

\bibitem{BH1}
Yu. Bilu and G. Hanrot, \emph{Solving Thue Equations of High Degree}, J. Number Th. \textbf{60} (1996), 373-392.

\bibitem{Bilu02}
Yu. Bilu, \emph{Baker's method and modular curves}, A Panorama of Number Theory or The View from Baker'S Garden
(edited by G. W\"ustholz), 73-88, Cambridge University Press, 2002.

\bibitem{BI11}
Yu. Bilu and M. Illengo, \emph{Effective Siegel's Theorem for Modular Curves}, Bull. London Math. Soc. \textbf{43} (2011), 673-688.

\bibitem{BP10}
Yu. Bilu and P. Parent, \emph{Runge's Method and Modular Curves}, Int. Math. Res. Notes \textbf{2011}(9) (2011), 1997-2027.

\bibitem{BP11}
Yu. Bilu and P. Parent, \emph{Serre¡¯s uniformity problem in the split Cartan case}, Ann. Math.  \textbf{173} (2011), 569-584.

\bibitem{BPR11}
Yu. Bilu, P. Parent and M. Rebolledo, \emph{Rational points on $X_0^+ (p^r)$}, Ann. Inst. Fourier, to appear;
	\textsf{arXiv:1104.4641}.

\bibitem{Bugeaud1}
Y. Bugead and K. Gy\H ory, \emph{Bounds for the solutions of unit equations}, Acta Arithm. \textbf{74} (1996), 67-80.

\bibitem{Dobrowolski1}
E. Dobrowolski, \emph{On a question of Lehmer and the number of irreducible factors of a polynomial},
Acta Arithm.  \textbf{34} (1979), 391-401.

\bibitem{Fr89}
E. Friedman, \emph{Analytic formulas for regulators of number fields}, Invent. Math. \textbf{98}
(1989), 599-622.

\bibitem{Frohlich}
A. Fr\"ohlich and T. Martin, \textit{Algebraic number theory}, Cambridge Studies in Advanced Mathematics \textbf{27}, Cambridge University Press, Cambridge, 1993.

\bibitem{GY}
K. Gy\H ory and K. Yu, \emph{Bounds for the solutions of $S$-unit equations and decomposable form equations}, Acta Arithm.
\textbf{123} (2006), 9-41.

\bibitem{Koblitz}
N.Koblitz, \textit{p-adic Numbers, p-adic Analysis, and Zeta-Functions}, GTM {\bf 58}, Second edition, Springer, New York, 1984.

\bibitem{KL81}
D.S. Kubert and S. Lang, \textit{Modular units}, Grund. Math.
Wiss. {\bf 244}, Springer, New York-Berlin, 1981.

\bibitem{La76}
S. Lang, \textit{Introduction to modular forms}, Grund. Math. Wiss. {\bf 222}, Springer, New York-Berlin, 1976.

\bibitem{Matveev}
E.M. Matveev, \emph{An explicit lower bound for a homogeneous rational linear form in the logarithms of algebraic numbers II},
Izv. Math. \textbf{64}(6) (2000), 1217-1269.

\bibitem{Shimura}
G. Shimura, \textit{Introduction to the arithmetic theory of automorphic functions}, Publ. Math. Soc. Japan 11, Iwanami Shoten,
Tokyo; Princeton University Press, Princeton, N.J., 1971.

\bibitem{Si29}
C.L. Siegel, \emph{$\ddot{\rm U}$ber einige Anwendungen diophantischer Approximationen}, Abh. Pr. Akad.
Wiss. (1929), no. 1. (=\textit{Ges. Abh.} I, 209-266, Springer, 1966.)

\bibitem{Si69}
C.L. Siegel, \emph{Absch$\ddot{\rm a}$tzung von Einheiten}, Nachr. Akad. Wiss. G¡§ottingen II. Math.-Phys. Kl. \textbf{9}
(1969), 71-86.

\bibitem{Waldschmidt2000}
M. Waldschmidt, \emph{Diophantine Approximation on Linear Algebraic Groups}, Springer-Verlag, 2000.

\bibitem{Wa}
L.C. Washington, \textit{Introduction to Cyclotomic Fields}, Springer-Verlag, 1982.

\bibitem{Yu07}
K. Yu, \emph{P-adic logarighmic forms and group varieties III}, Forum Math.  \textbf{19}(2007), 187-280.


\end{thebibliography}
\end{document}